\documentclass[12pt, a4paper]{amsart}
\usepackage[utf8]{inputenc}
\usepackage{amsxtra,amssymb,amsthm,amsmath,amscd,mathrsfs, epsfig, eufrak}
\usepackage{amscd, amsmath, mathrsfs, amssymb, amsthm, amsxtra, bbding, epsfig, eucal, eufrak, graphicx, latexsym, mathrsfs, mathbbol, bbold}
\usepackage[all]{xy}
\usepackage{tikz}

\usepackage[normalem]{ulem}
\usepackage{soul}
\usepackage{color}

\setstcolor{red}

\makeatletter
\@namedef{subjclassname@2010}{%
  \textup{2010} Mathematics Subject Classification}
\makeatother

\def\le{\leqslant}
\def\ge{\geqslant}

\theoremstyle{plain}
\newtheorem{theorem}{Theorem}
\newtheorem{proposition}{Proposition}[section]
\newtheorem{lemma}[proposition]{Lemma}

\theoremstyle{remark}

\numberwithin{equation}{section}

\addtocounter{footnote}{1}

\begin{document}

\vglue -5mm

\title[On a sum involving certain arithmetic functions and the integral part function]
{On a sum involving certain arithmetic functions and the integral part function}
\author{Jing Ma \& Huayan Sun}

\address{%
Jing Ma
\\
School of Mathematics
\\
Jilin University
\\
Changchun 130012
\\
P. R. China
}
\email{jma@jlu.edu.cn}

\address{%
Huayan Sun
\\
School of Mathematics
\\
Jilin University
\\
Changchun 130012
\\
P. R. China
}
\email{hysun19@mails.jlu.edu.cn}

\date{\today}

\begin{abstract}
In this short, we study sums of the shape $\sum_{n\leqslant x}{f([x/n])}/{[x/n]},$  where $f$ is Euler totient function $\varphi$, Dedekind function $\Psi$,
 sum-of-divisors function $\sigma$
or the alternating sum-of-divisors  function $\beta.$
We improve previous results when $f=\varphi$ and derive new estimates when $f=\Psi, f=\sigma$ and $f=\beta.$
\end{abstract}


\maketitle

\section{{Introduction}}

Let $x>10$ be a  real number and $n$ be a positive integer. The  integer part of $x/n$, namely $[x/n]$, is very important in analytic  number theory. For
example, the latest result for the Dirichlet divisor problem reads (see, Bourgain \cite{Bourgain})
\begin{equation*}
\sum_{n\leqslant x}\left[\frac{x}{n}\right]=x\log x+(2\gamma-1)x+O(x^{{517}/{1648}+o(1)}),
\end{equation*}
where $\gamma$ is the Euler constant.

Let $f$ be any  arithmetic function. It is interesting to study the asymptotic behaviour of the sum $$
S_{f}(x):=\sum_{n\leqslant x}f\left(\left[\frac{x}{n}\right]\right).
$$
This question  was first studied in \cite{BDHPS2019}, where a series of results were given when $f$ satisfies different conditions.
Wu \cite{Wu2020note} and Zhai \cite{Zhai2020} improved their results independently.
Several authors  studied the asymptotic formulas  for $S_{f}$ when $f$ equals some special arithmetic functions such as
\begin{equation*}
\begin{split}
& \tau(n):=\text{the divisor function,}\\
&\sigma(n):=\text{the sum-of-divisors function,}\\
&\beta (n):=\text{the alternating sum-of-divisors  function,}\\
&\omega(n):=\text{the number of distinct prime factors of $n$},\\
&\Omega(n):=\text{the number of  prime factors of $n$},\\
&\Lambda(n):=\text{the von Mangoldt function,}\\
&\varphi(n):=\text{the Euler totient function},\\
&\Psi(n):=\text{the Dedekind function},
\end{split}
\end{equation*}
etcetera.
(see \cite{O.Bordelles2020},\cite{BDHPS2019},\cite{LWY-2+},\cite{MaWuzhao},\cite{MaWu},\cite{MaSun}, for instance).
With the help of Vaughan identity and the technique of one-dimensional exponential sum, Ma and Wu \cite{MaWu}  proved
\begin{equation}\label{mangoldt}
S_{\Lambda}(x)
= x\sum_{m\ge 1}\frac{\Lambda(m)}{m(m+1)} + O_{\varepsilon}\big(x^{35/71+\varepsilon}\big).
\end{equation}
Using similar idea to the classical divisor function $\tau(n)$, Ma and Sun \cite{MaSun} showed that
\begin{equation}\label{d(n)}
S_{d}(x)
= x\sum_{m\ge 1}\frac{\ \tau(m)}{m(m+1)} + O_{\varepsilon}\big(x^{11/23+\varepsilon}\big).
\end{equation}
Subsequently, by using a result of Baker on 2-dimensional  exponential sums \cite[Theorem 6] {O.Bordelles2020},   Bordell\`es \cite{O.Bordelles2020}
sharpened the exponents in \eqref{mangoldt}-\eqref{d(n)} and also studied some new examples.
Very recently, Liu, Wu and Yang  used the multiple exponential sums to derive better results than \cite{O.Bordelles2020} (see \cite{LWY-1,LWY-2+}).

When $f(n)=\varphi(n)/n$, Bordell\`es-Dai-Heyman-Pan-Shparlinski \cite [Corollary 2.4] {BDHPS2019} derived that
$$
\sum_{n\leqslant x}\frac{\varphi([x/n])}{[x/n]}=x\sum_{m\ge 1}\frac{\varphi(m)}{m^{2}(m+1)}+O(x^{1/2}).
$$
Subsequently, Wu \cite{Wu2020note} refined arguments of \cite{BDHPS2019} and proved
\begin{equation}\label{1/3}
\sum_{n\leqslant x}\frac{\varphi([x/n])}{[x/n]}=x\sum_{m\ge 1}\frac{\varphi(m)}{m^{2}(m+1)}+O(x^{1/3}\log x).
\end{equation}
Stucky \cite{Stucky} further improved the result of $S_{\tau}$ and generalized Wu{'s} result in \cite{Wu2020note}.
%If $f$ satisfies $f(n)\ll n^{\alpha}(\log n)^{\theta}$ for some fixed $\alpha,\alpha \theta \geqslant 0$ with $\alpha<1$ He proved
He proved if $f(n)=\sum_{d|n}g(d)$ and $\sum_{d\leqslant x}|g(d)|\ll x^{\alpha}(\log x)^{\theta}$ for some $\alpha\in [0,1)$ and $\theta\geqslant 0$, then
\begin{equation}\label{stucky}
S_{f}(x)=\sum_{m\geqslant 1}\frac{f(m)}{m(m+1)}+O\left(x^{\frac{1+\alpha}{3-\alpha}}(\log x)^{\theta}\right),
\end{equation}
where the implied constant depends only on $\alpha$. (If $\alpha=0$, then $\theta$ should be replaced by $\max\left(1,\theta\right)$).
It is well known that
\begin{equation*}
\frac{\varphi(n)}{n}=\sum_{d|n}\frac{\mu(d)}{d}\qquad \text{and}\qquad
\sum_{d\leqslant x}\left|\frac{\mu(d)}{d}\right|\ll x^{\varepsilon},
\end{equation*}
 then applying \eqref{stucky}, we get
\begin{equation}\label{(1+x)/(3-x)}
\sum_{n\leqslant x}\frac{\varphi([x/n])}{[x/n]}=\sum_{m\geqslant 1}\frac{\varphi(m)}{m^{2}(m+1)}+O(x^{{1}/{3}+\varepsilon}),
\end{equation}

Motivated by \cite{LWY-1,LWY-2+},
we would like to improve \eqref{(1+x)/(3-x)}  by using three-dimensional exponential sums.
Noticing that four functions $\sigma, \varphi, \beta$ and $\Psi$ share many similarities such as multiplicative structures
and rates of growth (see \cite{MaWuzhao}), we can get a more general result.

\begin{theorem}\label{Thm1}
Let $f\in \{\sigma,\varphi,\beta,\Psi\}$.
For any $\varepsilon>0$, we have
\begin{equation*}
\sum_{n\leqslant x}\frac{f([x/n])}{[x/n]}
= x\sum_{m\ge 1}\frac{f(m)}{m^{2}(m+1)} + O\big(x^{3/13+\varepsilon}\big)
\end{equation*}
as $x\rightarrow \infty$.
\end{theorem}
We note that $\sum_{m\geqslant 1}\frac{\varphi(m)}{m^{2}(m+1)}\approx 0.78838 $.
To compare with \eqref{(1+x)/(3-x)}, we have $\frac{1}{3}\approx 0.33$ and $\frac{3}{13}\approx 0.23$.

\vskip 5mm

\section{Some Lemmas}

In this section, we will  cite two lemmas which will be needed in the next section.
The first one is due to Valaler (\cite [Theorem A.6] {GrahamKolesnik1991}).
\begin{lemma}\label{lem:1}
Let $\psi(t):=\{t\}-\frac{1}{2}$, where $\{t\}$ means the fractional part of real number $t$.
For~$x\geqslant{}1$ and ~$H\geqslant{}1$, we have
$$\psi(x)=-\sum_{1\leqslant{} |h|\leqslant{} H} \Phi\left(\frac{h}{H+1}\right)\frac{e(hx)}{2\pi ih}+R_H(x),$$where~$e(t):=e^{2\pi it},\Phi(t):=\pi t(1-|t|)\cos(\pi t)+|t|$, and the error term ~$R_H(x)$ satisfies
\begin{equation}\label{RH}
|R_H(x)|\leqslant{}\frac{1}{2H+2}\sum_{|h|\leqslant{} H}\left(1-\frac{|h|}{H+1}\right)e(hx).
\end{equation}
\end{lemma}

The second one is \cite [Proposition  3.1] {LWY-1}.

\begin{lemma}\label{lem:2}
Let $\alpha>0,\beta>0,\gamma>0$ and $\delta\in \mathbb{R}$ be some constants. For $X>0,H\geqslant 1,M\geqslant 1,$ and $N\geqslant 1,$ define
\begin{equation}\label{S_delta}
S_{\delta}=S_{\delta}(H,M,N):=\sum_{h\sim H}\sum_{m\sim M}\sum_{n\sim N}a_{h,n}b_{m}e\left(X\frac{M^{\beta }N^{\gamma}}{H^{\alpha}}\frac{h^{\alpha}}{m^{\beta}n^{\gamma}+\delta}\right),
\end{equation}
where $e(t):=e^{2\pi i t}$, the $a_{h,n}$ and $b_{m}$ are complex numbers such that $|a_{h,n}|\leqslant 1, b_{m}\leqslant 1$ and
$m\sim M$ means that $M<m\leqslant 2M$. For any $\varepsilon >0$ we have
\begin{equation}\label{S_delta-z}
S_{\delta}\ll \big((X^{\kappa}H^{2+\kappa}M^{1+\kappa+\lambda}N^{2+\kappa})^{1/(2+2\kappa)}+HM^{1/2}N+H^{1/2}MN^{1/2}+X^{-1/2}HMN\big)X^{\varepsilon}
\end{equation}
uniformly for $M\geqslant 1,N\geqslant 1,H\leqslant N^{\gamma-1}M^{\beta}$ and $0\leqslant \delta\leqslant 1/\varepsilon$, where $(\kappa,\lambda)$ is an exponent pair and the implied constants depends on $(\alpha,\beta,\gamma,\varepsilon)$ only.
\end{lemma}
Note that this lemma is actually a direct corollary of  \cite [Proposition  3.1] {LWY-1}.

\vskip 5mm

\section{A key estimate}

Now, we define
$$
\mathfrak{S}_{\delta}^{f}(x,D):=\sum_{D<d\leqslant 2D}\frac{f(d)}{d}\psi\left(\frac{x}{d+\delta}\right).
$$
Denote by $\mu(n)$ the M\"{o}bius function and by $\Omega(n)$ the number of all prime factors of $n$, respectively.
Using the relations:
$$
\frac{\sigma(n)}{n}=\sum_{d|n}\frac{1}{d},\quad
\frac{\varphi(n)}{n}=\sum_{d|n}\frac{\mu(d)}{d},\quad
\frac{\beta(n)}{n}=\sum_{d|n}\frac{(-1)^{\Omega(d)}}{d},
\quad \frac{\Psi(n)}{n}=\sum_{d|n}\frac{\mu^{2}(d)}{d},
$$
we  decompose $\mathfrak{S}_{\delta}^{f}(x,D)$ into bilinear form
\begin{equation*}
\begin{split}
&\mathfrak{S}_{\delta}^{\sigma}(x,D)=\sum_{D<dl\leqslant 2D}\frac{1}{l}\psi\left(\frac{x}{dl+\delta}\right),\\
&\mathfrak{S}_{\delta}^{\varphi}(x,D)=\sum_{D<dl\leqslant 2D}\frac{\mu(l)}{l}\psi\left(\frac{x}{dl+\delta}\right),\\
&\mathfrak{S}_{\delta}^{\beta}(x,D)=\sum_{D<dl\leqslant 2D}\frac{(-1)^{\Omega(l)}}{l}\psi\left(\frac{x}{dl+\delta}\right)\\
&\mathfrak{S}_{\delta}^{\Psi}(x,D)=\sum_{D<dl\leqslant 2D}\frac{\mu^{2}(l)}{l}\psi\left(\frac{x}{dl+\delta}\right).
\end{split}
\end{equation*}
The following estimation of $\mathfrak{S}_{\delta}^{f}(x,D)$ plays a key role in the proof of the main theorem.

\begin{proposition}\label{proposition}
Let $f\in \{\sigma,\varphi,\beta,\Psi\}$.
For any $\varepsilon>0$ and $0\leqslant \delta\leqslant \varepsilon^{-1}$, we have
\begin{equation}\label{key estimate}
\mathfrak{S}_{\delta}^{f}(x,D)\ll_{\varepsilon} (x^{2}D)^{1/12}x^{\varepsilon}
\end{equation}
uniformly for $x\geqslant 3$ and  $1\leqslant D\leqslant x^{8/11}$, where the  implied constant depends on $\varepsilon$ only.
\end{proposition}

\begin{proof}

We shall only bound $\mathfrak{S}_{\delta}^{\varphi}(x,D)$. The other cases 
can be treated in the same way.
Using Lemma \ref{lem:1} and splitting the interval of summation into the dyadic intervals, we can write
\begin{equation}\label{S_b+S_+}
\mathfrak{S}_{\delta}^{\varphi}(x,D)=-\frac{1}{2\pi i}\sum_{H{'}}\sum_{M}\sum_{N}(\mathfrak{S}_{\delta,\flat}^{\varphi}(H{'},M,N)+\overline{\mathfrak{S}_{\delta,\flat}^{\varphi}(H{'},M,N)})
+\sum_{M}\sum_{N}\mathfrak{S}_{\delta,\dag}^{\varphi}(M,N),
\end{equation}
where $H\leqslant D, MN\asymp D, M\geqslant D^{1/2}, a_{h}:=\frac{H{'}}{h}\Phi(\frac{h}{H+1})\ll 1$
and
\begin{equation*}
\mathfrak{S}_{\delta,\flat}^{\varphi}(H{'},M,N)=\frac{1}{H{'}}\sum_{h\sim H^{'}}a_{h}\mathop{{\sum}\,\,{\sum}}_{\substack{m\sim M   n\sim N\\ D<mn\le 2D}}\frac{\mu(m)}{m}e\left(\frac{hx}{mn+\delta}\right),
\end{equation*}
\begin{equation*}
\mathfrak{S}_{\delta,\dag}^{\varphi}(M,N)=\mathop{{\sum}\,\,{\sum}}_{\substack{m\sim M n\sim N\\ D<mn\le 2D}} \frac{\mu(m)}{m}R_{H}\left(\frac{x}{mn+\delta}\right).
\end{equation*}
(Here we can suppose $M\geqslant D^{1/2}$ in view of the symmetry of the variables $m$ and $n$.)

Firstly, we bound $\mathfrak{S}_{\delta,\flat}^{\varphi}(H{'},M,N)$.
We remove the extra multiplicative condition $D<mn\leqslant 2D$
 at the cost of a factor $\log x$, and then apply  \eqref{S_delta-z} with $\alpha=\beta=\gamma=1, (X,H,M,N)=(xH{'}/MN,H{'},M,N)$  and $(\kappa,\lambda)=(\frac{1}{2},\frac{1}{2})$  to get
\begin{equation*}
\begin{split}
\mathfrak{S}_{\delta,\flat}^{\varphi}(H{'},M,N)&\ll
\frac{x^{\varepsilon}}{H{'}}\sum_{m\sim M}\frac{1}{m}\sum_{h\sim H{'}}\sum_{n\sim N}e\left(\frac{hx}{mn+\delta}\right)\\
&\ll
\frac{x^{\varepsilon}}{H{'}M}\sum_{m\sim M}\sum_{h\sim H{'}}\sum_{n\sim N}e\left(\frac{hx}{mn+\delta}\right)\\
&\ll
\left((xN^{4}M^{-3})^{1/6}+NM^{-1/2}+(H{'}^{-1}N)^{1/2}(x^{-1}H{'}^{-1}MN^{3})^{1/2}\right){x^{\varepsilon}}\\
&\ll \left((x^{2}D)^{{1}/{12}}+D^{{1}/{4}}+(DH{'}^{-{2}})^{{1}/{4}}+(x^{-1}D^{2})^{{1}/{2}}\right)x^{\varepsilon}
\end{split}
\end{equation*}
where we used $1\leqslant H' \leqslant H\leqslant M,  MN\asymp D$ and $M\geqslant D^{1/2}$.
Noticing that $D\leqslant x^{{8}/{11}}$
we have
\begin{equation}\label{S_b}
\mathfrak{S}_{\delta,\flat}^{\varphi}(H{'},M,N)
\ll (x^{2}D)^{1/12}x^{\varepsilon}.
\end{equation}

Secondly, we bound $\mathfrak{S}_{\delta,\dag}^{\varphi}(M,N)$.
Using \eqref{RH} of Lemma \ref{lem:1}, we have
\begin{equation*}
\begin{split}
|\mathfrak{S}_{\delta,\dag}^{\varphi}(M,N)|&\leqslant \sum_{m\sim M}\frac{1}{m}\sum_{n\sim N}\left|R_{H}\left(\frac{x}{mn+\delta}\right)\right|\\
&\leqslant \frac{1}{2H+2}\sum_{0\leqslant |h|\leqslant H}\left(1-\frac{|h|}{H+1}\right)\sum_{m\sim M}\frac{1}{m}\sum_{n\sim N}e\left(\frac{hx}{mn+\delta}\right)\\
&\ll
\frac{1}{HM}\left(MN+\max_{1\leqslant H{'}\leqslant H}\left|\widetilde{\mathfrak{S}}_{\delta,\dag}^{\varphi}(H{'},M,N)\right|\right),
\end{split}
\end{equation*}
where $$\widetilde{\mathfrak{S}}_{\delta,\dag}^{\varphi}(H{'},M,N)=\sum_{h\sim H{'}}\sum_{m\sim M}\sum_{n\sim N}\left(1-\frac{|h|}{H+1}\right)e\left(\frac{hx}{mn+\delta}\right).$$
Clearly we can bound $\widetilde{\mathfrak{S}}_{\delta,\dag}^{\varphi}(H{'},M,N)$ in the same way as $\mathfrak{S}_{\delta,\flat}^{\varphi}(H{'},M,N)$
 and to get
%When $h\neq 0$, we can bound the triple sums as before. Thus
\begin{equation}\label{S_+}
\mathfrak{S}_{\delta,\dag}^{\varphi}(M,N)\ll (D^{{1}/{2}}H^{-1}+(x^{2}D)^{1/12})x^{\varepsilon}
\end{equation}
for $H\leqslant M\leqslant D$ and $MN\asymp D\leqslant x^{8/11}$.

Inserting \eqref{S_b} and \eqref{S_+} into \eqref{S_b+S_+} and taking $H=D^{{1}/{2}}$, we  can obtain the  required inequality.

\end{proof}

\vskip 5mm

\section{Proof of Theorem \ref{Thm1} }

In this section, let $f\in\{\sigma, \varphi, \beta, \Psi\}$
and
let $N_{f}\in [x^{\frac{3}{11}},x^{\frac{1}{3}})$ be a  parameter to be chosen later.
First we write
\begin{equation}\label{S3+S4}
\sum_{n\leqslant x}\frac{f([x/n])}{[x/n]}=S_{f}^{\dag}(x)+S_{f}^{\sharp}(x)
\end{equation}
with
\begin{equation*}
S_{f}^{\dag}(x)=\sum_{n\leqslant N_{f}}\frac{f([x/n])}{[x/n]},\qquad
S_{f}^{\sharp}(x)=\sum_{N_{f}<n \leqslant x}\frac{f([x/n])}{[x/n]}.
\end{equation*}

A. \textit{Bound of $S_{f}^{\dag}(x)$}

\vskip 1mm
We have
$f(n)\ll_{\varepsilon} n^{1+\varepsilon}$ for all $n\ge 1$ and any $\varepsilon>0$, thus
\begin{equation}\label{S3}
S_{f}^{\dag}(x)
\ll N_{f}x^{\varepsilon}.
\end{equation}

\vskip 1mm

B. \textit{Bound of $S_{f}^{\sharp}(x)$}

\vskip 1mm

In order to get the bound of $S_{f}^{\sharp}(x)$, we put $m=[x/n]$. Then$$
x/n-1<m\leqslant x/n \quad\Leftrightarrow \quad x/(m+1)<n\leqslant x/m.
$$
Thus
\begin{equation}\label{S4fenjie}
\begin{split}
S_{f}^{\sharp}(x)
&=\sum_{N_{f}<n\leqslant x}\frac{f([x/n])}{[x/n]}=\sum_{m\leqslant x/N_{f}}\frac{f(m)}{m}\sum_{x/(m+1)<n\leqslant x/m}1\\
&=\sum_{m\leqslant x/N_{f}}\frac{f(m)}{m}\left(\frac{x}{m}-\psi\left(\frac{x}{m}\right)-\frac{x}{m+1}+\psi\left(\frac{x}{m+1}\right)\right)\\
&=x\sum_{m\geqslant{} 1}\frac{f(m)}{m^{2}(m+1)}+ O(N_{f}x^{\varepsilon})+{R}_{0}^{f}(x)-{R}_{1}^{f}(x),
\end{split}
\end{equation}
where we have used the following bounds
\begin{equation*}
\begin{split}
&x\sum_{m>{x}/{N_{f}}}\frac{f(m)}{m^{2}(m+1)}\ll_{\varepsilon} x\sum_{m>{x}/{N_{f}}}\frac{m^{1+\varepsilon}}{m^{2}(m+1)}\ll N_{f}x^{\varepsilon},\\
&\sum_{m\leqslant N_{f}}\frac{f(m)}{m}\psi \left(\frac{x}{m+\delta}\right)\ll_{\varepsilon} \sum_{m\leqslant N}\frac{f(m)}{m}\ll N_{f}x^{\varepsilon}
\end{split}
\end{equation*}
and
\begin{equation*}
R_{\delta}^{f}(x)=\sum_{N_{f}<m\leqslant{} {x}/{N}_{f}}\frac{f(m)}{m}\psi \left(\frac{x}{m+\delta}\right).
\end{equation*}
Let ~$D_j:={x}/({2^jN_{f}}),$ we have $N_{f}\leqslant D_{j}\leqslant x/N_{f} \leqslant x^{{8}/{11}}$ for $0\leqslant j 
\leqslant{\log(x/N_{f}^{2})}/{\log 2}$.
Thus we can apply \eqref{key estimate} of Proposition \ref{proposition} to get
\begin{equation}\label{R_delta-f}
\begin{split}
R_{\delta}^{f}(x) &\leqslant \sum_{0\leqslant j \leqslant{\log(x/N_{f}^{2})}/{\log 2}}|\mathfrak{S}_{\delta}^{f}(x,D_{j})|\\
&\ll \sum_{0\leqslant j \leqslant{\log(x/N_{f}^{2})}/{\log 2}}\left(x^{2}D_{j}\right)^{1/12}x^{\varepsilon}\\
&\ll
(x^{3}N_{f}^{-1})^{1/12}x^{\varepsilon}.
\end{split}
\end{equation}
Inserting this into \eqref{S4fenjie} and taking $N_{f}=x^{3/13}$, we derive
\begin{equation}\label{S4}
S_{f}^{\sharp}(x)=x\sum_{m\geqslant{} 1}\frac{f(m)}{m^{2}(m+1)}+O(x^{3/13+\varepsilon}).
\end{equation}
Inserting \eqref{S3} with $N_{f}=x^{3/13}$ and \eqref{S4} into \eqref{S3+S4},
we get the required result.

\vskip 5mm

\noindent{\bf Acknowledgements}.
The authors are grateful for Professor  Jie Wu{’}s instructive talk and discussion.


\vskip 8mm

\begin{thebibliography}{150}


\bibitem{O.Bordelles2020}
O. Bordell\`es,
\emph{On certain sums of number theory},
arXiv:2009.05751v2 [math.NT] 25 Nov 2020.


\bibitem{BDHPS2019}
O. Bordell\`es, L. Dai, R. Heyman, H. Pan and I. E. Shparlinski,
\emph{On a sum involving the Euler function},
J. Number Theory {\bf 202} (2019), 278--297.





\bibitem{Bourgain}
J. Bourgain and N. Watt,
\emph{Mean square of zeta function, circle problem and divisor problem revisited},
arXiv:1709.04340v1 [math.AP] 13 Sep 2017.


%\bibitem{Carella}
%N. A. Carella,
%\emph{Average order of the Elur  phi function, the Dedekind psi function, the sum of divisors function, and the largest integer function},
%arXiv:2101.02248v2 [math.GM] 9 Apr 2021.

\bibitem{GrahamKolesnik1991}
S. W. Graham and G. Kolesnik,
\emph{Van Der Corput's method of exponential sums},
Cambridge University Press,
Cambridge, 1991.


\bibitem{LWY-1}
K. Liu, J. Wu and Z.-S. Yang,
\emph{  A variant of the prime number theorem},
arXiv:2105.10844v1 [math.NT] 23 May 2021.

\bibitem{LWY-2+}
K. Liu, J. Wu and Z.-S. Yang,
\emph{On some sums involving the integral part  function},
arXiv:2109.01382v1 [math.NT] 3 Sep 2021.




\bibitem{MaWuzhao}
J. Ma , J. Wu and F. Zhao,
\emph{On a generalisation of Bordell\`es-Dai-Heyman-Pan-Shparlinski{'} conjecture },
to appear in J. Number Theory.

\bibitem{MaWu}
J. Ma and J. Wu,
\emph{On a sum involving the Mangoldt function},
 Period. Math. Hung. {\bf 83} (2021), 39--48.




\bibitem{MaSun}
J. Ma and H.-Y. Sun,
\emph{On a sum involving the divisor function},
to appear in Period. Math. Hung.

\bibitem{Stucky}
J. Stucky,
\emph{The fractional sum of small arithmetic function},
arXiv:2106.14142v1 [math. NT] 27 Jun 2021.


\bibitem{Wu2020note}
J. Wu,
\emph{Note on a paper by Bordell\`es, Dai, Heyman, Pan and Shparlinski},
Period. Math. Hung. {\bf 80} (2020), 95--102.




\bibitem{Zhai2020}
W. Zhai,
\emph{On a sum involving the Euler function},
J. Number Theory {\bf 211} (2020), 199--219.


\bibitem{ZW}
F. Zhao and  J. Wu,
\emph{On the  sum involving the sum-of-divisors function},
J. of Math.,
Volume 2021.



\end{thebibliography}
\end{document}